\documentclass[10pt]{article}
\usepackage{amsmath}
\usepackage{amssymb}
\usepackage{amsthm}
\usepackage{graphicx} 
\usepackage{appendix}

\newtheorem{theorem}{Theorem}[section]
\newtheorem{lemma}[theorem]{Lemma}
\newtheorem{proposition}[theorem]{Proposition}
\newtheorem{corollary}[theorem]{Corollary}

\newtheorem{definition}[theorem]{Definition}
\newtheorem{hypothesis}[theorem]{Hypothesis}

\newcommand{\R}{{\bf R}}
\newcommand{\C}{{\bf C}}

\newcommand{\cA}{{\cal A}}
\newcommand{\cB}{{\cal B}}
\newcommand{\cE}{{\cal E}}

\newcommand{\cH}{{\cal H}}
\newcommand{\cL}{{\cal L}}

\newcommand{\charac}[3]{\Psi_{ {}_{\small[{#1},{#2}\small], {#3} } } }
\newcommand{\tcharac}[3]{\widetilde{\Psi}_{ {}_{\small[{#1},{#2}\small], {#3} } } }
\newcommand{\dcharac}[4]{\Psi^{#4}_{ {}_{\small[{#1},{#2}\small], {#3} } } }

\newcommand{\lam}{\lambda}

\newcommand{\cint}{\iint\limits_{\C}}
\newcommand{\rint}{\int\limits_{-\infty}^{\infty}}

\newcommand{\bra}[1]{\left( #1 \right)}
\newcommand{\sbra}[1]{\small( #1 \small)}

\newcommand{\norm}[1]{\Vert #1 \Vert}

\newcommand{\beq}{\begin{equation}}

\newcommand{\kik}[1]{C^{\infty}_{c}\bra{#1}}
\newcommand{\kiko}[1]{C^{\infty}_{0}\bra{#1}}

\newcommand{\sx}{\langle x \rangle}

\newcommand{\algebra}{\cA}
\newcommand{\ealgebra}{\hat{\algebra}}

\begin{document}

\title{\bf The Spectral Mapping Theorem}
\author{Narinder S Claire}
\date{}

\maketitle

\begin{abstract}
We give a direct non-abstract proof of the Spectral Mapping Theorem for the Helffer-Sj\"ostrand functional calculus for linear operators on Banach spaces with real spectra 
and consequently give a new non-abstract direct proof for the Spectral Mapping Theorem for self-adjoint operators on Hilbert spaces. Our exposition is closer
in spirit to the proof by explicit construction of the existence of the Functional Calculus given by Davies. We apply an extension theorem of Seeley to derive a functional
calculus for semi-bounded operators. 
\\ \\
{\bf AMS Subject Classification :  47A60\\
Keywords : Functional Calculus, Spectral Mapping Theorem, Spectrum, }  
\end{abstract}

\section{Introduction}
The Helffer-Sj\"ostrand formula was established in \cite{hs} in the following proposition
\begin{proposition}[\cite{hs} Proposition 7.2]\label{hsformula}
 Let $H$ be a self-adjoint operator (not necessarily bounded) on a Hilbert space $\cH$. Suppose $f$ is in $C^\infty_0\sbra{\R}$ and $\tilde{f}$ in 
$C^\infty_0\sbra{\C}$ is an extension of $f$ such that $\frac{\partial\tilde{f}}{\partial\bar{z}}=0$ on $\R$. Then we have
\begin{equation}\label{e:hs}
f\bra{H} = \frac{i}{2\pi}\cint \frac{\partial \tilde{f}}{\partial \overline{z}}
\bra{z-H}^{-1}d\bar{z}\wedge dz =-\frac{1}{\pi}\cint \frac{\partial \tilde{f}}{\partial \overline{z}}
\bra{z-H}^{-1}dxdy
\end{equation}
where $L\sbra{C}$ is the Lebesgue measure on $C$.
\end{proposition}
The existence of the functional calculus was assumed by the authors. Davies\cite{Da1}
showed that the formula (equation \ref{hsformula}) yielded a new approach to \underline{constructing} the functional calculus for 
linear operators on Banach spaces under the following hypothesis 
\begin{hypothesis}\label{e:hypothesis}
$H$ is a closed densely defined operator on a Banach space $\cB$ with spectrum $\sigma\sbra{H}\subseteq \R$. The resolvent operators
$\sbra{z-H}^{-1}$ are defined and bounded for all $z\notin\R$ and 
\begin{equation}
\norm{\sbra{z-H}^{-1}}\leq c |Im\,z|^{-1}\bra{\frac{\langle z \rangle}{|Im\,z|}}^\alpha
\end{equation}
for some $\alpha\geq 0$ and all $z\notin\R$, where $\langle z \rangle := \sbra{1+|z|^2}^{\frac{1}{2}}$.
\end{hypothesis}
His functional calculus, for operators on Banach spaces, was defined for an algebra of slow decreasing smooth functions. Davies\cite{Da1} pointed out that a functional calculus based upon almost analytic extensions was also constructed by Dyn'kin\cite{dk}. However, the two 
approaches were quite different and that Davies' approach was more appropriate for differential operators.\\
The Spectral Mapping theorem for the Helffer-Sj\"ostrand functional calculus was also independently proved by B\'atkai and Fa\v{s}anga \cite{Ba}. They
applied methods from abstract functional analysis and their primary tool was an existing abstract Spectral Mapping Theorem from the theory of Banach algebras  :
\begin{theorem}[\cite{Ba} Theorem 4.1]
Let $\cB_1$ be a commutative, semisimple, regular Banach algebra, $\cB_2$ be a Banach algebra with unit,
$\Theta : \cB_1 \rightarrow   \cB_2$ be a continuous algebra homomorphism and $a \in \cB_1$. Then 
$$ \sigma_{\cB_1}\bra{\Theta\bra{a}} = \overline{\hat{a}\bra{Sp\bra{\theta}}} \quad \text{ where } Sp\bra{\Theta}:= \cap_{b\in Ker \Theta}\,Ker\, \hat{b} $$ 
and $\hat{}$ denotes the Gelfand transform.
\end{theorem}
Our exposition, part of the Ph.D thesis referred to in the introduction of \cite{Ba}, takes a very non-abstract and direct approach to the proof.
In particular an existing spectral mapping is not assumed. Our sole ingredients, supplementing the tools provided in Davies\cite{Da1}, are the very elementary observations:
\begin{itemize}
\item It is possible to join two non-zero points in $\C$ smoothly without passing through the origin.
\item (\cite{Da3} Problem 8.1.11) If $H$ is a closed operator and $\lambda$ lies in the topolgical boundary of the spectrum of $H$ then
 for every $\epsilon>0$ there is a vector $v$ with length $1$ such that $\norm {Hv-\lambda v}<\epsilon$
\item Stokes Formula has similarities to the Cauchy Integral Forumla.
\end{itemize}
A compelling argument for a direct proof that does not rely on spectral mapping results from the theory of Banach algebras follows from the claim by Davies\cite{Da1} 
that all the calculations in the construction of his functional calculus can all be carried out at a Banach algebra level rather than at an operator level, provided one
has a resolvent family in the algebra satifying the obvious analogue of hypothesis \ref{e:hypothesis}.\\ \\
In the last part of our exposition we derive a functional calculus for operators with spectra bounded on one side. 
Our main tool here is an extension operator of Seeley :
$$ \cE : C^{\infty}[0,\infty) \longrightarrow C^{\infty}\bra{\R}$$

\subsection{Functional Calculus}
We summarize some of the main aspects of the Helffer-Sj\"ostrand functional calculus presented in Davies\cite{Da2} and some properties of the algebra. 
Let $\psi_{{}_{a,\epsilon}}$ be a smooth function such that
\begin{equation*}
\psi_{{}_{a,\epsilon}}\bra{x}: = \begin{cases} 1 & \text{ if } x\geq a\\ 0 & \text{ if } x\leq a-\epsilon  \end{cases}
\end{equation*}
Then given an interval $\left[a,b\right]$ we define the approximate characteristic function $\charac{a}{b}{\epsilon} $
$$\charac{a}{b}{\epsilon} \bra{x} = \psi_{{}_{a,\epsilon}}\bra{x}- \psi_{{}_{b+\epsilon,\epsilon}}\bra{x}$$
which has support $\left[a-\epsilon,b+\epsilon \right]$ and is equal to  $1$ in $\left[a,b\right]$ and is smooth.
\begin{definition}
For $\beta \in \R$ let
$ S^{\beta}$ to be the set of all complex-valued smooth functions defined on $\R$ where for every $n$
there is a positive constant $c_n$ such that
where $$|\frac{d^nf}{dx^n}|\, \leq\, c_n \sx^{\beta-n}$$
We then define the Algebra $\cA:= \bigcup\limits_{\beta<0} S^{\beta} $
\end{definition}
\begin{lemma}[Davies \cite{Da1,Da2}]
$\cA$ is an algebra under point-wise multiplication.
For any $f$ in $\cA$ the expression 
\begin{equation}
\norm{f}_n := \sum\limits_{r=0}^{n}\rint |\frac{d^{r}f}{dx^r}|\,\sx^{r-1} dx
\end{equation}
defines a norm on $\cA$ for each $n$.
Moreover $\kiko{\R}$ is dense in $\cA$ with this norm.
\end{lemma}
\begin{lemma}
The function $\sx^{\beta}$ is in $\cA$ for each $\beta <0$
\end{lemma}
\begin{proof}
The statement  follows from the observations that if $\beta <0 $ and $m\geq n$ then
$$ x^n\sx^{\beta-m}\leq \sx^{\beta}$$
and $$\frac{d\bra{x^n\sx^{\beta-m}}}{dx}=nx^{n-1}\sx^{\beta-m} +2\bra{\beta-m}x^{n+1}\sx^{\beta-m-2}$$
\end{proof}
\begin{lemma}\label{e:second}
Let $s\in \R$. If $f$ is in $\cal A$ then the function
$$g_s \bra{x} := \begin{cases} \frac{f\bra{x}-f\bra{s}}{x-s} \quad x \neq s\\
                                           f'\bra{s} \quad x=s
                          \end{cases}$$
is also in $\cal A$
\end{lemma}
\begin{proof}
When $|x-s|$ is large then $$\frac{1}{|x-s|}\leq c_s\langle x\rangle^{-1}$$
for some $c_s>0$. Moreover  
$$
g_s^{\bra{r}}\sbra{x}=\sum\limits_{m=0}^r
c_r f^{\bra{m}}\sbra{x}\,\bra{x-s}^{m-r-1}\, + \,cf\bra{s}\bra{x-s}^{-r-1}
$$
and
$$ \lim\limits_{x\rightarrow s}g_s^{\bra{m}}\sbra{s}=f_s^{\bra{m+1}}\sbra{s}$$

\end{proof}
\begin{lemma}\label{lem:stable}
If $f\in S^{\beta}$ for $\beta<0$ and $g\in S^0$ then $fg \in \cA$ 
\end{lemma}
\begin{proof}
$$
|\sbra{fg}^{\bra{r}}\sbra{x}| 
\leq c_r\sum\limits_{m=0}^r |g^{\bra{r-m}}\sbra{x}|\,|f^{\bra{m}}\sbra{x}| 
\leq c_{r,\phi}\sx^{\beta-r}
$$
\end{proof}

The following concept of almost analytic extensions is due to H\"ormander\cite[p63]{ho}.
\begin{definition}\label{almostanalytic}
Let $\tau\bra{x,y}$ be a smooth function such that 
$$\tau\bra{x,y} := \begin{cases} 1 & \text{ if } |y|\leq \langle x \rangle \\ 0 & \text{ if } |y|\geq 2\langle x \rangle   \end{cases}$$
Then given $f \in \algebra$ we define an almost
analytic extension $\tilde{f}$ as
\begin{equation}\label{e:ae}
\tilde{f}\bra{x,y} := \bra{\sum\limits_{r=0}^n \frac{d^{r}f\bra{x}}{dx^r}\frac{\bra{iy}^r}{r!}}\tau\bra{x,y}
\end{equation}
Moreover we define 
\begin{equation}\label{e:get}
\frac{\partial \tilde{f}}{\partial \overline{z}} := \frac{1}{2}\bra{\frac{\partial \tilde{f}}{\partial x}+
i\frac{\partial \tilde{f}}{\partial y}}
\end{equation}
\end{definition}

The following lemma establishes the construction of the new functional calculus
\begin{lemma}[Davies\cite{Da1}]\label{e:important}
Let $f\in\algebra$ then define 
\begin{equation}\label{e:hsD}
f\bra{H} := -\frac{1}{\pi}\cint \frac{\partial \tilde{f}}{\partial \overline{z}}
\bra{z-H}^{-1}dxdy
\end{equation}
where $\tilde{f}$ is an almost-analytic version of $f$ as defined in definition \ref {almostanalytic}. Then 
\begin{enumerate}
\item If $n>\alpha$ then subject to hypothesis \ref{e:hypothesis} the integral \eqref{e:hsD} is norm
convergent for all $f$ in $\cA$ and
$$ \norm{f\bra{H}}\leq c\norm{f}_{n+1}$$
\item The operator $f\bra{H}$ is independent of $n$ and the cut-off function $\tau$, subject to $n>\alpha$
\item If $f$ is a smooth function of compact support disjoint from the spectrum of $H$ then $f\bra{H}=0$
\item If $f$ and $g$ are in $\cA$ then $\bra{fg}\bra{H}=f\bra{H}g\bra{H}$ 
\item If $z \not\in \R$ and $g_z\bra{x} := \bra{z-x}^{-1}$ for all $x \in \R$ then
$g_z \in \cA$ and $g_z\bra{H}=\bra{z-H}^{-1}$
\end{enumerate}
\end{lemma}

\subsection{Preliminaries}
\begin{definition}
Given $z$, $\omega$ in $\C$ we define the curve $\Gamma$ in the complex plane
$$\Gamma\sbra{z,\omega,\alpha} :=\bra{\sbra{1-\alpha}|z|+\alpha|\omega|}
e^{i\bra{1-\alpha}Arg\bra{z}+i\alpha Arg\bra{\omega}}$$
where $\alpha\in[0,1]$ and $z,w \in \C$
\end{definition}
The important property of $\Gamma$ is that it is able to connect two non-zero points in the complex plane without intersection with the origin.
\begin{theorem}\label{deeep} Let $\lambda\in \C$.
If $f$ is a smooth complex valued function in the interval $[a,b]$ where $f\bra{a}\neq \lambda$ and $f\bra{b}\neq \lambda$ then there is a smooth function 
$h$ in $C^\infty\bra{[a.b]}$ such that $$\{x\in[a,b] : h\bra{x}=\lambda\} \text { is empty }$$
Moreover ${f-g}$ and all derivatives of ${f-g}$ vanish at $a$ and $b$.  
\end{theorem}
\begin{proof}Let 
\begin{equation}
g\bra{x}:=\Gamma\bra{f\bra{a}-\lam,f\bra{b}-\lam,
\frac{x-a}{b-a}}+\lam
\end{equation}
Since $f$ is continuous we know there is an $0<\epsilon<\frac{b-a}{2}$ such that 
$$\{x\in[a,b]/\bra{a+\epsilon,b-\epsilon} : f\bra{x}=\lambda\} = \emptyset$$
Then we can define 
$$h := \bra{1-\charac{a+\epsilon}{b-\epsilon}{\epsilon}}f  + \charac{a+\epsilon}{b-\epsilon}{\epsilon}g $$
\end{proof}

\begin{lemma}\label{gotcha}
Given  $f\in\algebra$, let $\lambda$ be a non-zero point in $\C$ and let $A_\lambda := \{x : f\bra{x}=\lambda\}$\\
If $A_\lambda \cap \sigma\bra{H}$ is empty then there is a function $h\in\algebra$ such that 
$h\bra{x}\neq \lambda$ for all $x\in\R$ and $$h\bra{H}=f\bra{H}$$

\end{lemma}
\begin{proof}
If $A_\lambda$ is empty then we put $h=f$.\\ If $A_\lambda$ is not empty then 
$A_\lambda$ is a compact subset of $\rho\bra{H}$. Moreover
$A_\lambda$ can be covered by a finite set of closed disjoint intervals $[a_i,b_i]$ which are also subsets of $\rho\bra{H}$.
By applying theorem \ref{deeep} to each interval we can find a function $h$ in $\algebra$ such $$h\sbra{x}=f\sbra{x} \text{ for all }x \in \sigma\bra{H}$$ 
and $h\sbra{x}\neq\lambda \text{ for all }x \in \R$.
Moreover since $\sbra{f-h}$ has compact support in $\rho\bra{H}$ then it follows from 
lemma \ref{e:important}(iii) that $h\bra{H}=f\bra{H}$.
\end{proof}

\section{Bounded Operators}
We let $B$ be a bounded operator satisfying hypothesis (\ref{e:hypothesis}). Moreover let $$ u := \sup \sigma\bra{B} \quad \text { and } \quad l := \inf\sigma\bra{B} $$

\begin{lemma}\label{compactifyf}
For any $f\in\algebra$ and $\epsilon>0$ $$f\charac{l'}{u'}{\epsilon}\bra{B} = f\bra{B} $$  
where $l'\leq l$ and $u'\geq u$
\end{lemma}
\begin{proof}
Suppose $f$ has compact support then $f-f\charac{l'}{u'}{\epsilon}$ has compact support disjoint from the spectrum of $B$ hence
by lemma \ref{e:important}(iii) the statement of the lemma is true for functions in $C^\infty_0\sbra{\R}$. The statement for all $f\in\algebra$ follows from
density of $C^\infty_0\sbra{\R}$ in $\algebra$. 
\end{proof}

\begin{lemma}\label{calcforBO}
Let $f\in \algebra$. If $\epsilon>0$ and 
$$D_{{}_\epsilon}:= \{z : |z-\tfrac{u+l}{2}|<\tfrac{u-l}{2}+\epsilon\}\quad \text{ and } 
\quad \partial D_{{}_\epsilon}:= \{z : |z-\tfrac{u+l}{2}|=\tfrac{u-l}{2}+\epsilon\}$$ then 
$$f\bra{B} = \frac{1}{2\pi i}\int_{\partial D_{{}_\epsilon}} \tilde{f}\bra{z}\bra{z-B}^{-1} dz -\frac{1}{\pi}\int_{D_{{}_\epsilon}} \frac{\partial \tilde{f}}{\partial \overline{z}}
\bra{z-B}^{-1}dxdy$$
\end{lemma}
\begin{proof}
By lemma \ref{compactifyf} we can assume that $f$ has compact support in $\left[l-\epsilon, u+\epsilon\right]$ \\
If $R>\tfrac{u-l}{2}+\epsilon$ and $A_R$ is the annulus $\{z : \tfrac{u-l}{2}+\epsilon<|z-\tfrac{u+l}{2}|<R\}$ then 
$$\int_{|z-\tfrac{u+l}{2}|<R} \frac{\partial \tilde{f}}{\partial \overline{z}} \bra{z-B}^{-1}dxdy = \int_{A_R} \frac{\partial \tilde{f}}{\partial \overline{z}}
\bra{z-B}^{-1}dxdy +  \int_{D_{{}_\epsilon}} \frac{\partial \tilde{f}}{\partial \overline{z}}
\bra{z-B}^{-1}dxdy$$
Applying Stokes' theorem 

$$ \int_{A_R} \frac{\partial \tilde{f}}{\partial \overline{z}}
\bra{z-B}^{-1}dxdy = \frac{1}{2i}\int_{|z-\tfrac{u+l}{2}|=R} \tilde{f}\bra{z-B}^{-1}dz - \frac{1}{2i}\int_{\partial D_{{}_\epsilon} } \tilde{f}\bra{z-B}^{-1}dz$$ 
and letting $R$ be large enough for $\tilde{f}$ to vanish on $\{z:|z-\tfrac{u+l}{2}|=R\}$ completes the proof. 
\end{proof}

\begin{lemma}\label{oneforBO}
Let $\epsilon>0$. If  $l'< l$ and $u'>u$ then $$\charac{l'}{u'}{\epsilon}\bra{B} = 1  $$ 
\end{lemma}
 
\begin{proof}
Let $0<\delta<1$ and define $\Omega$ as the open rectangle $$\{z\in \C : |Re\, z-\tfrac{u'+l'}{2}|<\tfrac{u'-l'}{2} , \quad |Im\, z| < \delta\}$$
Using a similar argument to that given in the proof of lemma \ref {calcforBO} we see that 
$$\charac{l'}{u'}{\epsilon}\bra{B} = \frac{1}{2\pi i}\int_{\partial \Omega} \tcharac{l'}{u'}{\epsilon}\sbra{z,\overline{z}}\bra{z-B}^{-1} dz -
\frac{1}{\pi}\int_{\Omega} \frac{\partial \tcharac{l'}{u'}{\epsilon}}{\partial \overline{z}}\bra{z-B}^{-1}dxdy$$
When $l'\leq x \leq u'$  then $\charac{l'}{u'}{\epsilon}\sbra{x} = 1$. Moreover when $l'\leq x \leq u'$  then $\dcharac{l'}{u'}{\epsilon}{\bra{n}}\sbra{x} = 0$ for all $n>0$. 
Recalling definition (\ref{e:ae}) we can see that $$\tcharac{l'}{u'}{\epsilon}\sbra{z,\overline{z}}=1 \quad \text { for all } z\in \overline{\Omega_0}$$ 
hence 
$$\charac{l'}{u'}{\epsilon}\bra{B} = \frac{1}{2\pi i}\int_{\partial \Omega} \bra{z-B}^{-1} dz $$
and we conclude with an application of Cauchy's integral formula.
 
\begin{figure}[h!]
 \centering
 \includegraphics[scale=.5 ]{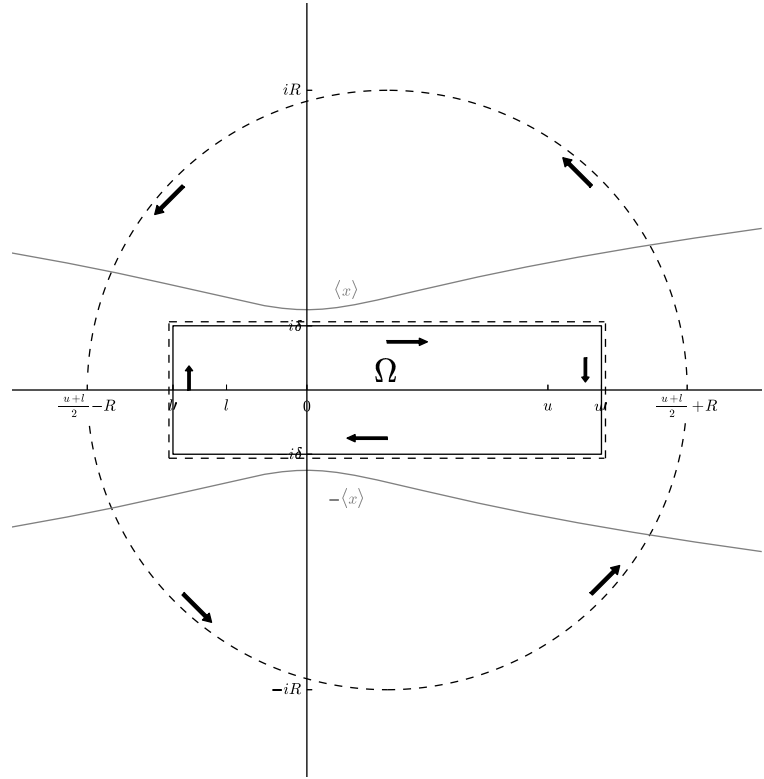}
 \caption{Integral domain for lemma \ref {oneforBO}}
 \label{fig:integral}
\end{figure}
\end{proof}

\section{Enlargement of $\cA$}
We extend the algebra $\cA$ of slow decaying functions in a trivial but necessary way.
\begin{definition}Let
$$\hat{\cA} := \{\bra{z,f}:z \in \C , f \in \algebra\}$$
where for each $x\in \R$ we define $$\bra{z,f}\sbra{x} := z + f\sbra{x}$$
Moreover we define point-wise addition and multiplication:
\begin{equation*}
\begin{split}
\bra{\omega,f}\circ \bra{z,g}:=\bra{\omega z,\omega g+zf+fg} \\
\bra{\omega,f}+ \bra{z,g}:=\bra{\omega+z,f+g}
\end{split}
\end{equation*}
It is clear that $\bra{1,0}$ the multiplicative identity and $\bra{0,0}$ the additive identity are in $\ealgebra$ and 
the algebra is closed under these operations.\\
For any $z\in\C$ we will denote $\bra{z,0}\in\ealgebra$ simply by $z$.\\
Given $\phi=\bra{z,f}\in\ealgebra$, let $$\pi_{{}_{\cA,\phi}} := f \quad \text{ and } \quad \pi_{{}_{C,\phi}} := z$$
and let
\begin{equation*}\label{e:normrel}
\norm{\phi}_n:=|\pi_{{}_{C,\phi}}| + \norm{\pi_{{}_{\cA,\phi}}}_n
\end{equation*}
\end{definition}
\begin{definition}
We have the  extended functional calculus. For $\phi \in \ealgebra$ let
$$ \phi\sbra{H} := \pi_{{}_{\cA,\phi}}\sbra{H} + \pi_{{}_{C,\phi}}I$$
along with the implied norm
\begin{eqnarray*}
\norm{\phi\bra{H}}&:=&|\pi_{{}_{C,\phi}}| + \norm{\pi_{{}_{\cA,\phi}}\bra{H}} \\
& \leq & |\pi_{{}_{C,\phi}}| + \norm{\pi_{{}_{\cA,\phi}}}_{n+1}\\
&=& \norm{\phi}_{n+1}
\end{eqnarray*}
\end{definition}
\begin{definition}
For $\phi \in \ealgebra$ let:
$$ \mu\bra{\phi} := \frac{1}{\pi_{{}_{C,\phi}} + \pi_{{}_{\cA,\phi}}}-\frac{1}{\pi_{{}_{C,\phi}} }$$
\end{definition}
\begin{lemma}\label{lem:thisone}
If $\phi \in \ealgebra$\ and $-\pi_{{}_{C,\phi}}$ is not in  $\overline{Ran\sbra{\pi_{{}_{\cA,\phi}}}}$
then $\mu\bra{\phi}$ is in $\cA$ and $$\phi^{-1} = \bra{\frac{1}{\pi_{{}_{C,\phi}} },\mu\bra{\phi}}$$
\end{lemma}
\begin{proof}
By re-writing
$$\mu\bra{\phi}=\frac{1}{\pi_{{}_{C,\phi}} +\pi_{{}_{\cA,\phi}}}-\frac{1}{\pi_{{}_{C,\phi}} } =
\frac{-\pi_{{}_{\cA,\phi}}}{\pi_{{}_{C,\phi}} \bra{\pi_{{}_{C,\phi}} +\pi_{{}_{\cA,\phi}}}}$$
then it is routine exercise in differentiation  to show that
 $$\frac{-1}{\pi_{{}_{C,\phi}} \bra{\pi_{{}_{C,\phi}} +\pi_{{}_{\cA,\phi}}}}$$ is in $S^0$.
Then since $\pi_{{}_{\cA,\phi}}$ is in $\cA$, lemma \ref{lem:stable} implies the statement.
\end{proof}
\begin{corollary}\label{inversecor}
 Given $\phi\in \ealgebra$ and  $\lambda\in\C$ such that $\phi\bra{x}\neq \lambda$ for all $x\in\R$ then $$\bra{\phi-\lambda}^{-1}\in\ealgebra$$
\end{corollary}

\section{Spectral Mapping Theorem}

\begin{lemma}\label{auto}
If $\phi$ is in $\ealgebra$ then
$$ \sigma \bra{\phi\sbra{H}} \subseteq \overline{Ran\bra{\phi}}$$
\end{lemma}
\begin{proof}
Given $\lam\in\C$ which is not in $\overline{Ran\bra{\phi}}$ we have by by corollary \ref{inversecor}
$$ \bra{\phi-\lambda}^{-1} \in \ealgebra$$
hence $\bra{\phi\bra{H}-\lam}^{-1}$ exists and is bounded and therefore  $\lam \not\in \sigma\bra{\phi\bra{H}}$.
\end{proof}

\begin{lemma}\label{unbounded}
If $\phi$ is in $\ealgebra$ then
$$\sigma\bra{\phi\bra{H}}\subseteq\phi\bra{\sigma\sbra{H}}\cup \{\pi_{{}_{\C,\phi}}\}$$ 
\end{lemma}
\begin{proof}
Let $\lambda\in\C$ be such that  $\lambda\neq \pi_{{}_{\C,\phi}}$ and let $$A_\lambda = \{x : \phi\sbra{x}=\lambda\}$$
If $A_\lambda \cap \sigma\sbra{H} = \emptyset$ then by lemma \ref{gotcha} we have that there is function $h$ in $\algebra$ such that $$h\sbra{x}=\pi_{{}_{\algebra,\phi}}\sbra{x} \text{ for all }x \in \sigma\bra{H}$$ 
and $$h\sbra{x}\neq\lambda-\pi_{{}_{\C,\phi}} \text{ for all }x \in \R$$
moreover $$h\bra{H} = \pi_{{}_{\algebra,\phi}}\bra{H}$$
If $\theta := \bra{\pi_{{}_{\C,\phi}} , h}\in\ealgebra$ it follows from 
lemma \ref{e:important} we have $$\phi\bra{H}=\theta\bra{H}$$
Since $\lambda\notin \overline{Ran\bra{\theta}}$, the statement of the lemma follows from lemma \ref{auto}.
\end{proof}

\begin{lemma}\label{bbounded}
Let $\phi\in\ealgebra$. If $H$ is bounded and $$\{x : \phi\sbra{x}=\pi_{{}_{\C,\phi}}\}\cap \sigma\sbra{H} \, \text{ is empty }$$ 
then $\pi_{{}_{\C,\phi}}\notin\sigma\bra{\phi\bra{H}}$
\end{lemma}
\begin{proof}
Let $ u := \sup \sigma\bra{H}$ and $l := \inf\sigma\bra{H} $.\\
Let $0<\epsilon\ll1$ such that $\pi_{{}_{\algebra,\phi}}$ is not zero on $\left[ l-\epsilon,l \right]$ and on $\left[ u,u+\epsilon \right]$.\\ 
Then let  $ u' := u+\epsilon$ and $l' := l-\epsilon$.\\
The set  $$\{x\in\left[l',u'\right] : \pi_{{}_{\cA,\phi}}\sbra{x}=0\}$$
can be covered by a finite number of disjoint intervals $\left[a_i,b_i\right]$ which are all disjoint from $\sigma\bra{H}$ and are all in $\left[l',u' \right]$.
Applying lemma \ref {deeep}  to each $\left[a_i,b_i\right]$ we can find a function $f\in \algebra$ such that 
$$\{x\in\left[l',u'\right] : f\sbra{x}=0\}=\emptyset$$ and $f=\pi_{{}_{\cA,\phi}}$ for all x in $\R/\left[l',u'\right]$.\\
Let $g$ be any function in $\algebra$ such that $g\bra{x}=\frac{1}{f\bra{x}}$ for all $x\in \left[l',u'\right]$\\ 
By lemma \ref{e:important}(iii) we have 
$$ \pi_{{}_{\cA,\phi}}\bra{H} g\bra{H} = f\bra{H} g\bra{H} $$
and by lemma \ref{compactifyf} we have 
$$ f\bra{H} g\bra{H} =  \sbra{fg\charac{l'}{u'}{\epsilon}}\sbra{H} = \charac{l'}{u'}{\epsilon}\sbra{H}$$
hence by lemma \ref{oneforBO} we have 
$$ \pi_{{}_{\cA,\phi}}\bra{H} g\bra{H} = 1 $$ and consequently $$ \bra{\pi_{{}_{\C,\phi}} - \phi\sbra{H}}\, g\sbra{H} = 1 $$ 
\end{proof}

\begin{theorem}
If $\phi$ in $\ealgebra$ then $\sigma\bra{\phi\bra{H}}\subseteq\overline{\phi\bra{\sigma\sbra{H}}}$
\end{theorem}
\begin{proof}
If $H$ is unbounded then $\overline{\phi\bra{\sigma\sbra{H}}} = \phi\bra{\sigma\sbra{H}}\cup \{\pi_{{}_{\C,\phi}}\} $
and the theorem follows from lemma \ref{unbounded}. If $H$ is bounded and there is an $x\in \sigma\bra{H}$ such that 
$ \phi\bra{x}= \pi_{{}_{\C,\phi}}$ then $\overline{\phi\bra{\sigma\sbra{H}}} = \phi\bra{\sigma\sbra{H}}\cup \{\pi_{{}_{\C,\phi}}\} $ 
and again the theorem follows from \ref{unbounded}. If $H$ is bounded and 
$$ \phi\bra{x}\neq\pi_{{}_{\C,\phi}}$$  for all $x\in \sigma\bra{H}$  then $\overline{\phi\bra{\sigma\sbra{H}}} = \phi\bra{\sigma\sbra{H}}$ 
by lemmas \ref{unbounded} and \ref{bbounded}.
\end{proof}

\begin{lemma}\label{e:third}
Given $s\in\R$ and a function $f\in \algebra$, let $k_s \bra{x}:=\bra{1,-\frac{s+i}{x+i}}\in \ealgebra$ and let the function $g_s$ be defined as in lemma \ref{e:second} then
 $$\bra{f\bra{H}-f\bra{s}}\bra{H+i}^{-1} = g_s\bra{H}k_s \bra{H}$$
\end{lemma}
\begin{proof}
This statement follows directly from the functional calculus and the observation
$$\bra{-f\bra{s},f\bra{x}}\bra{0,\bra{x+i}^{-1}}=\bra{0,\frac{f\bra{x}-f\bra{s}}{x-s}}
\bra{1,-\frac{s+i}{x+i}}$$ 
\end{proof}

\begin{theorem}
Let $f$ be a function in $\algebra$ then 
$$ \overline{f\bra{\sigma\bra{H}}} \subseteq \sigma\bra{f\bra{H}}$$
\end{theorem}
\begin{proof} 
We observe the identity 
\begin{equation}\label{loid}
 H-x= \bra{H+i}-\bra{x+i} =\bra{1-\bra{x+i}\bra{H+i}^{-1}}\bra{H+i} = k_x \bra{H}\bra{H+i}
\end{equation} for some $x\in\R$.\\
Let $s\in\R$. Suppose there is a sequence of unit length vectors $\{v_m\}\subset Dom\bra{H}$  such that
$\lim\limits_{m\rightarrow\infty}\bra{H-s}v_m = 0$.
Using identity (\ref{loid} ) we have $\lim\limits_{m\rightarrow\infty}g_s\bra{H}k_s\bra{H}\bra{H+i}v_m=0$.
By applying lemma \ref{e:third} we can conclude that $\lim\limits_{m\rightarrow\infty}\bra{f\bra{H}-f\bra{s}}v_m=0$.\\
The accumulation points of $f\bra{\sigma\bra{H}}$ are in $\sigma\bra{f\bra{H}}$ since the latter is closed.
\end{proof}
\begin{corollary}
Let $\phi$ be a function $\ealgebra$ then 
$$ \overline{\phi\bra{\sigma\bra{H}}} \subseteq \sigma\bra{\phi\bra{H}}$$
\end{corollary}

\section{Self-Adjoint Operators}
We now assume that $H$ is self-adjoint and $\cB$ is a Hilbert space.\\
The following theorem of Davies extends the Helffer-Sj\"ostrand functional calculus to $C_0\bra{\R}$ for self-adjoint operators.
\begin{theorem}[Davies\cite{Da1} Theorem 9]\label{safc}
The functional calculus may be extended to a map from $f\in C_0\bra{\R}$ to 
$f\bra{H}\in\cL\bra{\cB}$ with the follwoing properties:
\begin{enumerate}
\item $f\rightarrow f\bra{H}$ is an algebra homomorphism.
\item $\overline{f}\bra{H} = f\bra{H}^*$
\item $\norm{f\bra{H}}\leq \norm{f}_\infty$
\item If $z \not\in \R$ and $g_z\bra{x} := \bra{z-x}^{-1}$ for all $x \in \R$ then $g_z\bra{H}=\bra{z-H}^{-1}$
\end{enumerate}
Moreover the functional calculus is unique subject to these conditions.
\end{theorem}
\begin{lemma}
If $f\in C_0\bra{\R}$ then $$\overline{f\bra{\sigma\bra{H}}}\subseteq \sigma\bra{f\bra{H}}$$
\end{lemma}
\begin{proof}
This is a consequence of the density of $\algebra$ in $C_0\bra{\R}$.
By the Stone-Weierstrass theorem the linear subspace $$\{\sum\limits_{i=1}^{n} \tfrac{\lambda_i}{x-\omega_i} : \lambda\in\C \quad \omega_i\notin\R\}$$ 
is dense in $C_0\bra{\R}$.
If $f_\epsilon\in\algebra$ is close to $f$ and if $v\in\cB$ is of length 1 then 
$$ \norm{f\bra{H}v - f\bra{s}v}\leq \norm{f\bra{H}-f_\epsilon\bra{H}} + \norm{f_\epsilon\bra{H}v-f_\epsilon\bra{s}v} + \norm{f_\epsilon-f}_{\infty} $$
The statement then follows from lemma \ref{safc}(iii)
\end{proof}
\begin{lemma}
If $f\in C_0\bra{\R}$ then $$\sigma\bra{f\bra{H}} \subseteq \overline{f\bra{\sigma\bra{H}}}$$
\end{lemma}
\begin{proof}
Let $$f :=  \sum\limits_{i=1}^{\infty} \tfrac{\lambda_i}{x-\omega_i} \quad \text{ and } \quad f_n := \sum\limits_{i=1}^{n} \tfrac{\lambda_i}{x-\omega_i} $$
Suppose $\lambda\in \C$ is not in the closure of $f\bra{\sigma\bra{H}}$. Then there is $\delta>0$ such that $$\inf\limits_{s\in\sigma\bra{H}} |f\bra{s}-\lambda|=\delta$$
Also for all large enough $n$ we have $\norm{f_n-f}_\infty <\tfrac{\delta}{2}$.  Then from $$ |f\bra{s} -f_n\bra{s} +f_n\bra{s}-\lambda| >\delta$$
we can deduce that $$ |f_n\bra{s}-\lambda| >\delta - \norm{f_n-f}_\infty $$ hence
$$\inf\limits_{s\in\sigma\bra{H}} |f_n\bra{s}-\lambda|>\tfrac{\delta}{2}$$
and $\lambda\notin\sigma\bra{f_n\bra{H}}$.\\
From the identity $$\norm{\bra{f\bra{H}-\lambda}\bra{f_n\bra{H}-\lambda}^{-1}-1} = \norm{\bra{f\bra{H}-f_n\bra{H}}\bra{f_n\bra{H}-\lambda}^{-1}} $$
we can deduce that $\lambda\notin\sigma\bra{f\bra{H}}$.
\end{proof}

\section{Functional Calculus for Semi-Bounded Operators}
We modify our main hypothesis (\ref{e:hypothesis}) by assuming that the spectrum of $H$ is bounded below and without 
loss of generality $\sigma\bra{H} \subseteq [0,\infty)$.\\ 
We introduce a new ring
of functions $\cA^+$
\begin{definition} 
$S^{\beta}_+$ is the set of smooth functions on $\R^+\cup\{0\}$ with the same decaying property
as $S^\beta$ that is for every $n$ there is positive constant $c_n$ such that $$|\frac{d^nf}{dx^n}|\,
\leq\, c_n \sx^{\beta-n}$$
Then $\cA^+$ is defined appropriately and similarly we define the Banach space $\cA^+_n$ with norm
\begin{equation}
\norm{f}_{\cA^+_n} := \sum\limits_{r=0}^{n}\int\limits_0^{\infty} |\frac{d^{r}f}{dx^r}|\sx^{r-1} dx
\end{equation}
\end{definition}

We present a theorem due to Seeley \cite{se} which gives a linear extension operator
for smooth functions from the half space to the whole space. 
\begin{theorem}[Seeley's Extension Theorem]
There is a linear extension operator
$$ \cE : C^{\infty}[0,\infty) \longrightarrow C^{\infty}\bra{\R}$$
such that for all $x>0$
$$\bra{\cE f}\bra{x}=f\bra{x} $$
\end{theorem}
The extension operator is continuous for many topologies including uniform convergence of each derivative.
The proof of the theorem relies on the following lemma.
\begin{lemma}[\cite{se}]\label{lem:seeley}
There are sequences $\{a_k\},\, \{b_k\}$ such that
\begin{enumerate}
\item $b_k <0$
\item $\sum\limits_{k=0}^{\infty}|a_k||b_k|^n<\infty$\label{p:two}
 for all non-negative integers $n$ \label{p:two}
\item $\sum\limits_{k=0}^{\infty}a_k \bra{b_k}^n =1$ for all non-negative integers $n$
\item $b_k \rightarrow -\infty$
\end{enumerate}
\end{lemma}
The proof to Seeley's extension theorem is by construction and it is informative to give explicitly the extension. First we need a to define two linear operators.
\begin{definition}
Given  $f \in \cA^+$, $\phi \in \cA$ and real $a$ we define two operators on $\cA^+$,
$$\bra{T_af}\bra{x}=f\bra{ax}$$
$$\bra{S_{\phi}f}\bra{x}=\phi\bra{x}f\bra{x}$$
\end{definition}
\begin{proof}[Proof of Seeley's Extension Theorem]
Let $\phi \in \kik{\R}$ such that
$$\phi\bra{x}=\begin{cases} 1 & x \in [0,1] \\ 0 & x \geq 2 \\ 0 & x \leq -1 \end{cases} $$
Then define  $\cE$ such that 
$$ \bra{\cE f}\bra{x} := \begin{cases} \sum\limits_{k=0}^\infty a_k \bra{T_{b_k}S_{\phi}f}\bra{x} &x <0\\
                                                f\bra{x} & x\geq 0
\end{cases} $$
\end{proof}

\begin{lemma}
If  $a>1$ then $\norm{T_a}_{\cA^+_n \rightarrow \cA^+_n}\leq a^n$
\end{lemma}
\begin{proof}
Follows from 
\begin{eqnarray*}
\norm{T_af}_{\cA^+_n}&=&\sum\limits_{r=1}^n\int\limits_0^{\infty} |\frac{d^rf\bra{ax}}{dx^r}|\sx^{r-1}dx \leq 
\sum\limits_{r=1}^na^r\int\limits_0^{\infty} |\frac{d^rf\bra{x}}{dx^r}|\sx^{r-1}dx
\end{eqnarray*}
\end{proof}
\begin{lemma}If $\phi\in\cA$ then 
$S_{\phi}$ is a bounded operator with respect
 to each norm $\norm{\,}_{\cA^+_n}$
\end{lemma}
\begin{proof}
A simple application of Leibnitz gives
$$ \frac{d^r\bra{\phi\bra{x}f\bra{x}}}{dx^r}=\sum\limits_{m=0}^rc_r\frac{d^{r-m}\bra{\phi\bra{x}}}{dx^{r-m}}\,
\frac{d^m\bra{f\bra{x}}}{dx^m}$$ then
\begin{eqnarray*}
|\frac{d^r\bra{\phi\bra{x}f\bra{x}}}{dx^r}|&\leq &c_r\sum\limits_{m=0}^rd_{r-m,\phi}\,\sx^{\beta-\bra{r-m}}
\frac{d^m\bra{f\bra{x}}}{dx^m}\\
&\leq & c_{r,\phi}\sum\limits_{m=0}^r\,\sx^{m-r}\frac{d^m\bra{f\bra{x}}}{dx^m}
\end{eqnarray*}
we integrate to give
\begin{eqnarray*}
\int\limits_{0}^{\infty}
|\frac{d^r\bra{\phi\bra{x}f\bra{x}}}{dx^r}|\sx^{r-1}dx & \leq &
c_{r,\phi}\sum\limits_{m=0}^r\int\limits_{0}^{\infty}|\frac{d^m\bra{f\bra{x}}}{dx^m}|\sx^{m-1}dx\\
& = & c_{r,\phi}\norm{f}_{\cA^+_r}
\end{eqnarray*}
and hence we have our estimate
\begin{eqnarray*}
\norm{S_{\phi}f}_n & = & \sum\limits_{r=0}^n\int\limits_{0}^{\infty}
|\frac{d\bra{\phi\bra{x}f\bra{x}}^r}{d^rx}|\sx^{r-1}dx\\
& \leq & c_{n,\phi}\sum\limits_{r=0}^n \norm{f}_{\cA^+_r} \\
&\leq &c_{n,\phi}\norm{f}_{\cA^+_n}
\end{eqnarray*}
\end{proof}
\begin{theorem}
Seeley's Extension Operator is a bounded operator on each of the normed vector spaces $\cA_n^+$
\end{theorem}
\begin{proof}
\begin{eqnarray*}
\norm{\cE f}_{\cA_n} &=& \sum\limits_{r=0}^n \int\limits_{-\infty}^{\infty}|\frac{d^r\bra{\cE f}}{dx^r}|\sx^{r-1}dx\\
&=&\sum\limits_{r=0}^n \int\limits_0^{\infty}|\frac{d^rf\bra{x}}{dx^r}|\sx^{r-1}dx +
\sum\limits_{r=0}^n \int\limits_{-\infty}^{0}|\sum\limits_0^{\infty}a_k\frac{d^r\bra{\phi\bra{b_kx}f\bra{b_kx}}}{dx^r}|\sx^{r-1}dx
\\ &=&\norm{f}_{\cA^+_n} + \norm{\sum\limits_{k=0}^{\infty}a_k T_{-b_k}S_{\phi}f}_{\cA^+_n}\\
& \leq & \norm{f}_{\cA^+_n} +\sum\limits_{k=0}^{\infty}|a_k|\,\norm{S_{\phi}}\,\norm{|T_{-b_k}}\norm{f}_{\cA^+_n}\\
& \leq &\norm{f}_{\cA^+_n} +\bra{\sum\limits_{k=0}^{\infty}|a_k|\,|b_k|^n}\,c_{n,\phi}\norm{f}_{\cA^+_n}
\end{eqnarray*}
and hence the extension operator is continuous.
\end{proof}
If $f$ and $g$ are elements of $\cA$ such that $f|_{[0,\infty]}=g|_{[0,\infty]}$ and the spectrum of $H$ is $[0,\infty)$ then it is not
necessary that $supp\bra{f-g}\cap\sigma\bra{H}$ is empty, since $supp\bra{f-g}\cap\sigma\bra{H}=\{0\}$
is possible. 
\begin{lemma}
If $f$ is a smooth function on $\R$ of compact support such that
$$supp\bra{f}=[-a,0]$$ and $H$ is an operator satisfying our modified
hypothesis with $\sigma\bra{H}\subseteq [0,\infty]$
then $$f\bra{H}=0$$
\end{lemma}
\begin{proof}
Let $\epsilon \in \bra{0,1}$ and define $$f_{\epsilon}\bra{x}:=f\bra{x+\epsilon}$$
so that $supp\bra{f_{\epsilon}}=[-\bra{a+\epsilon},-\epsilon]$.\\
By lemma \ref{e:important}(iii) $ f_{\epsilon}\bra{H}=0$.
For all $n$ there are constants $c_n\geq 0$ such that 
$$\norm{\frac{d^nf}{dx^n}-\frac{d^nf_{\epsilon}}{dx^n}}_{\infty}\leq c_n\epsilon$$
then 
\begin{eqnarray*}
\norm{f\bra{H}}&=&\norm{f\bra{H}-f_{\epsilon}\bra{H}}\\
&\leq & \sum\limits_{r=0}^n \int\limits_{-\bra{a+1}}^0
|\frac{d^rf\bra{x}}{dx^r}-\frac{d^rf_{\epsilon}\bra{x}}{dx^r}|\sx^{r-1}dx\\
&\leq &\sum\limits_{r=0}^n \epsilon c_r \int\limits_{-\bra{a+1}}^0 \sx^{r-1}dx\\
&=& \epsilon k_{n,f}
\end{eqnarray*}
hence our result.
\end{proof}
\begin{corollary}\label{last}
If $f$ and $g$ are in $\cA$ such that $f|_{[0,\infty]}=g|_{[0,\infty]}$
and $\sigma\bra{H}\subseteq [0,\infty]$ then $f\bra{H}-g\bra{H}=0$
\end{corollary}
\begin{theorem}
If $H$ satisfies our modified hypothesis with spectrum $\sigma\bra{H} \subseteq [0,\infty)$
then there is a functional calculus $\gamma_H :\cA^+\rightarrow {\cL}\bra{\cB}$ 
and moreover for all $f\in \cA^+\cap\cA$ $$\gamma_H\bra{f}=-\frac{1}{\pi}\cint \frac{\partial \tilde{f}}{\partial 
\overline{z}}
\bra{z-H}^{-1}dxdy$$
\end{theorem}
\begin{proof}
Let $f^+ \in \cA^+$, then by Seeley's Extension Theorem there exists an extension $f\in\cA$.
We define $\gamma_H\bra{f^+}:=f\bra{H}$. This definition is independent of the particular 
extension by corollary \ref{last}. The functional analytic properties are inherited from the extension.
\end{proof}  
\begin{theorem}[Refinement of Theorem 10 of \cite{Da1}]
Let $n\geq1$ be an integer and $t>0$.\\If we denote the operator $\gamma_H\bra{e^{-s^nt}}$ by $e^{-H^nt}$ then
$$e^{-H^n\bra{t_1+t_2}} = e^{-H^n{t_1}} e^{-H^n{t_2}} $$ for all $n\geq1$ and $0<t\leq 1$

\end{theorem}

\begin{center}
 {\Large{\bf{ Acknowledgements}}}
 \end{center}
This research was funded by an EPSRC Ph.D grant 95-98 at Kings College, London. I am very grateful to E. Brian Davies
for giving me this problem, his encouragement since and for continuing to be a mentor in Mathematics long after 
having finished supervising my Ph.D. I am indebted to Anita for all her support.

\vskip 0.3in
Narinder Claire \newline
Global Equities \& Commodity Derivatives Quantitative Research \newline
BNP Paribas London \newline
10 Harewood Avenue\newline
London\newline
NW1 6AA \newline
e-mail: narinder.claire@uk.bnpparibas.com
\end{document}